\documentclass[12pt]{amsart}

\usepackage{amssymb}
\usepackage{amscd}

\title[The number of Galois points]{On the number of Galois points for a plane curve in characteristic zero}
\author{Satoru Fukasawa}

\address{Department of Mathematical Sciences, Faculty of Science, Yamagata University, Kojirakawa-machi 1-4-12, Yamagata 990-8560, Japan}
\email{s.fukasawa@sci.kj.yamagata-u.ac.jp} 
\thanks{The author was partially supported by JSPS KAKENHI Grant Number 25800002.}

\newtheorem{theorem}{Theorem}[section]
\newtheorem{proposition}[theorem]{Proposition} 

\newtheorem{lemma}[theorem]{Lemma}
\newtheorem{fact}[theorem]{Fact}
\newtheorem{conjecture}[theorem]{Conjecture}

\theoremstyle{definition}
 
\newtheorem{remark}[theorem]{Remark}

\begin{document}
\begin{abstract}
For a plane curve, a point on the projective plane is said to be Galois if the projection from the point as a map from the curve to a line induces a Galois extension of function fields. 
We present upper bounds for the number of Galois points, if the genus is greater than zero. 
If the curve is not an immersed curve, then we have at most two Galois points. 
If the degree is not divisible by two nor three, then the number of outer Galois points is at most three. 
As a consequence, a conjecture of Yoshihara is true in these cases. 
\end{abstract}
\maketitle

\section{Introduction}  
In 1996, H. Yoshihara introduced the notion of the {\it Galois point} (\cite{miura-yoshihara, yoshihara1}). 
Let $C \subset \Bbb P^2$ be an irreducible plane curve of degree $d \ge 4$ over an algebraically closed field $K$ of characteristic zero, let $C_{\rm sm}$ be the set of all smooth points of $C$, and let $K(C)$ be its function field. 
A point $P \in C_{\rm sm}$ (resp. $P \in \mathbb{P}^2 \setminus C$) is said to be inner (resp. outer) Galois for $C$, if the function field extension $K(C)/\pi_P^*K(\Bbb P^1)$ induced by the projection $\pi_P: C \dashrightarrow \Bbb P^1$ from $P$ is Galois. 
The number of inner (resp. outer) Galois points is denoted by $\delta(C)$ (resp. $\delta'(C)$).  
It is interesting to determine $\delta(C)$ or $\delta'(C)$.  

If $C$ is smooth, then Yoshihara and Miura (\cite{miura-yoshihara, yoshihara1}) showed that $\delta(C)=0, 1$ or $4$ (resp. $\delta'(C)=0,1$ or $3$), and $\delta(C)=4$ (resp. $\delta'(C)=3$) if and only if $C$ is projectively equivalent to the curve defined by 
$$ X^3Z+Y^4+Z^4=0 \ \ (\mbox{resp. } X^d+Y^d+Z^d=0).  $$ 
If $d$ is prime and $C$ is not rational, then Duyaguit and Miura showed that $\delta'(C) \le 3$ (\cite{duyaguit-miura}). 
If $d-1$ is prime, then Miura gave an upper bound related to $\delta(C)$ (\cite{miura}). 
If $C$ is rational and $d \ne 12, 24, 60$, then Yoshihara showed that $\delta'(C) \le 3$ (\cite{yoshihara2}). 
The present author gave an upper bound for $\delta(C)$ (\cite{fukasawa2}); however, the bound is not sharp (in characteristic zero). 
Yoshihara conjectured the following (\cite{open}). 

\begin{conjecture}
For any irreducible plane curve $C$ of degree $d \ge 4$, 
$$ \delta(C) \le 4, \mbox{ and } \ \delta'(C) \le 3. $$
\end{conjecture}

In this study, we show that the conjecture is true in many cases. 
Let $r:\hat{C} \rightarrow C$ be the normalization, and let $g$ be the genus of $\hat{C}$. 

\begin{theorem} \label{main1}
Let $C \subset \Bbb P^2$ be an irreducible plane curve of degree $d \ge 4$, and let $g \ge 1$. 
If the morphism $r:\hat{C} \rightarrow \mathbb{P}^2$ is not unramified, that is, there exists a point $\hat{Q} \in \hat{C}$ such that the differential map of $r$ at $\hat{Q}$ is zero, then $$\delta(C)+\delta'(C) \le 2.$$ 
\end{theorem}

By virtue of Theorem \ref{main1}, to find a bound for $\delta(C)$ or $\delta'(C)$, we have only to consider the case where $r:\hat{C} \rightarrow \mathbb{P}^2$ is unramified. 

\begin{theorem} \label{main2}
Let $C \subset \Bbb P^2$ be an irreducible plane curve of degree $d \ge 4$ (resp. $d \ge 3$), and let $g \ge 1$.  
Assume that $d-1$ (resp. $d$) is not divisible by two.  
Then, 
$$\delta(C) \le 5 \ \ (\mbox{resp. } \delta'(C) \le 5). $$ 
Furthermore, we have the following. 
\begin{itemize}
\item[(a)] If $d-1$ (resp. $d$) is not divisible by three, then $\delta(C) \le 3$ (resp. $\delta'(C) \le 3$). 
\item[(b)] If $g < (d-2)/2$ (resp. $g<(d+2)/2$), then $\delta(C) \le 3$ (resp. $\delta'(C) \le 3$). 
\item[(c)] If $g <2d-4$ (resp. $g<2d+1$), then $\delta(C) \le 4$ (resp. $\delta'(C) \le 4$). 
\item[(d)] If $2 \le g \le (d-1)^2/84$ (resp. $2 \le g \le d^2/84$), then $\delta(C) \le 1$ (resp. $\delta'(C) \le 1$). 
\end{itemize} 
\end{theorem}

\begin{remark} 
\begin{itemize}
\item[(a)] Curves under the condition that $d-1$ is a prime number larger than $3$ (resp. that $d=4$) satisfy the assumption of (a) (resp. of (c)) in Theorem \ref{main2}. 
\item[(b)] Curves under the condition that $d$ is a prime number larger than $3$ (resp. that $d=3$) satisfy the assumption of (a) (resp. of (b)) in Theorem \ref{main2}. 
\end{itemize} 
\end{remark}

\section{Preliminaries}
Let $(X:Y:Z)$ be a system of homogeneous coordinates of the projective plane $\Bbb P^2$. 
When $P \in C_{\rm sm}$, the (projective) tangent line at $P$ is denoted by $T_PC \subset \Bbb P^2$. 
For a projective line $\ell \subset \Bbb P^2$ and a point $P \in C \cap \ell$, $I_P(C, \ell)$ is the intersection multiplicity of $C$ and $\ell$ at $P$.  
The line passing through points $P$ and $R$ is denoted by $\overline{PR}$, when $R \ne P$, and the projection from a point $P \in \Bbb P^2$ by $\pi_P$. 
The projection $\pi_P$ is represented by $Q \mapsto \overline{PQ}$.  
Let $r: \hat{C} \rightarrow C$ be the normalization, and let $g$ be the genus of $\hat{C}$.  
We write $\hat{\pi}_P=\pi_P \circ r$. 
The ramification index of $\hat{\pi}_P$ at $\hat{Q} \in \hat{C}$ is denoted by $e_{\hat{Q}}$. 
If $Q=r(\hat{Q}) \in C_{\rm sm}$, then $e_{\hat{Q}}$ is denoted also by $e_Q$.   
It is not difficult to check the following.  

\begin{lemma} \label{index}
Let $P \in \Bbb P^2$, and let $\hat{Q} \in \hat{C}$ with $r(\hat{Q})=Q \ne P$. 
Then for $\hat{\pi}_P$ we have the following. 
\begin{itemize}
\item[(1)] If $P \in C_{\rm sm}$, then $e_P=I_P(C, T_PC)-1$.  
\item[(2)] If $h$ is a linear polynomial defining $\overline{PQ}$ around $Q$, then $e_{\hat{Q}}={\rm ord}_{\hat{Q}}r^*h$. 
In particular, if $Q$ is smooth, then $e_Q =I_Q(C, \overline{PQ})$.  
\end{itemize} 
\end{lemma}

The order sequence of the morphism $r:\hat{C} \rightarrow \Bbb P^2$ is $\{0, 1, 2\}$ (see \cite[Ch. 7]{hkt}, \cite{stohr-voloch}). 
If $\hat{Q} \in \hat{C}$ is a non-singular branch, i.e., there exists a line defined by $h=0$ with ${\rm ord}_{\hat{Q}}r^*h=1$, then there exists a unique tangent line at $Q=r(\hat{Q})$ defined by $h_{\hat{Q}}=0$ such that ${\rm ord}_{\hat{Q}}r^*h_{\hat{Q}} \ge 2$.  
The order ${\rm ord}_{\hat{Q}}r^*h_{\hat{Q}}$ of the tangent line $h_{\hat{Q}}=0$ at $\hat{Q}$ is denoted by $\nu_{\hat{Q}}$. 
If $\nu_{\hat{Q}}>2$, then we call the point $\hat{Q}$ (or $Q=r(\hat{Q})$ if $Q \in C_{\rm sm}$) a {\it flex}.  
The set of all non-singular branches is denoted by $\hat{C}_0 \subset \hat{C}$.   
We recall the following fact (see \cite[Theorem 1.5]{stohr-voloch}).  

\begin{fact}[Count of flexes] \label{flexes} We have
$$ \sum_{\hat{Q} \in \hat{C}_0} (\nu_{\hat{Q}}-2) \le 3(2g-2)+3d. $$
\end{fact}

On a Galois covering of curves, the following holds in general (see \cite[III. 7.2, 8.2]{stichtenoth}). 

\begin{fact} \label{Galois covering} 
Let $\theta: C \rightarrow C'$ be a Galois covering of degree $d$, and let $P \in C$. 
The ramification index at $P$ is denoted by $e_P$, and the stabilizer subgroup of $P$ by $G(P)$. 
Then we have the following. 
\begin{itemize}    
\item[(1)] The order of $G(P)$ is equal to $e_P$ for any point $P \in C$. 
\item[(2)] If $\theta(P)=\theta(Q)$, then $e_P=e_Q$. 
\item[(3)] The index $e_P$ divides the degree $d$. 
\end{itemize} 
\end{fact}

\section{Proof} 
Whenever we consider a Galois point $P$, we assume that $P$ is inner or outer Galois, that is, $P \in C_{\rm sm} \cup (\mathbb{P}^2 \setminus C)$.  
For a Galois point $P$, the Galois group is denoted by $G_P$. 
If $P \in C_{\rm sm}$ (resp. $P \in \mathbb{P}^2 \setminus C$), then the order $|G_P|$ is equal to $d-1$ (resp. $d$). 
We can consider $G_P$ as a subgroup of the automorphism group ${\rm Aut}(\hat{C})$. 
The following holds. 

\begin{lemma} \label{two Galois}
Let $P_1, P_2 \in \mathbb{P}^2$ be Galois points with $P_1 \ne P_2$. 
Then, $G_{P_1} \cap G_{P_2} =\{1\}$. 
\end{lemma}

\begin{proof}
For points $P_1, P_2 \in \mathbb{P}^2 \setminus C$, the assertion holds due to \cite[Lemma 7]{fukasawa1}. 
The proof for the case where $P_1, P_2 \in C_{\rm sm}$ is similar. 
If $P_1 \in C_{\rm sm}$ and $P_2 \in \mathbb{P}^2 \setminus C$, then the assertion is obvious, since the orders $|G_{P_1}|$ and $|G_{P_2}|$ are coprime. 
\end{proof}

Using Lemma \ref{two Galois}, and the well-known Hurwitz bound $84(g-1)$ for the order of the automorphism group of any smooth curve with genus $g\ge 2$, we have the following. 

\begin{proposition} \label{automorphism bound}
Assume that $g \ge 2$. 
If $\delta(C) \ge 2$ (resp. $\delta'(C) \ge 2$), then we have the inequality 
$$ (d-1)^2 \le 84(g-1) \ \ (\mbox{resp. } d^2 \le 84(g-1)). $$
\end{proposition}

\begin{proof}
Let $P_1, P_2$ be distinct Galois points, and let $G_{P_1}, G_{P_2} \subset {\rm Aut}(\hat{C})$ be the Galois groups. 
The order of the subgroup generated by $G_{P_1}$ and $G_{P_2}$ is at least $(d-1)^2$ (resp. $d^2$), by Lemma \ref{two Galois}. 
By the Hurwitz bound, we have the conclusion.   
\end{proof}

Hereafter, we assume that $g \ge 1$. 
The following fact is well-known (\cite[Lemma 11.44]{hkt}). 
\begin{lemma} \label{fixed}
Let $G$ be a subgroup of the automorphism group ${\rm Aut}(\hat{C})$, and let $\hat{Q} \in \hat{C}$. 
If $\sigma(\hat{Q})=\hat{Q}$ for any $\sigma \in G$, then $G$ is a cyclic group. 
\end{lemma}

By using this fact, we have the following. 

\begin{lemma} \label{unramified1}
Let $P_1$ and $P_2$ be distinct Galois points, let $\hat{Q} \in \hat{C}$ with $Q=r(\hat{Q}) \ne P_1, P_2$, and let $h_1$ and $h_2$ be linear polynomials defining $\overline{P_1Q}$ and $\overline{P_2Q}$ around $Q$ respectively.
If ${\rm ord}_{\hat{Q}}r^*h_1={\rm ord}_{\hat{Q}}r^*h_2$, then ${\rm ord}_{\hat{Q}}r^*h_1={\rm ord}_{\hat{Q}}r^*h_2=1$. 
\end{lemma}

\begin{proof}
By the assumption, $m:={\rm ord}_{\hat{Q}}r^*h_1={\rm ord}_{\hat{Q}}r^*h_2$. 
Assume that $m \ge 2$.
By Lemma \ref{index}(2) and Fact \ref{Galois covering}(1), there exist subgroups $G_1$ of $G_{P_1}$ and $G_2$ of $G_{P_2}$ of order $m$ respectively such that $\sigma(\hat{Q})=\hat{Q}$ for any $\sigma \in G_1 \cup G_2$. 
Let $G$ be the group generated by subgroups $G_1$ and $G_2$. 
Since $G_1 \cap G_2 \subset G_{P_1}\cap G_{P_2}=\{1\}$ by Lemma \ref{two Galois}, $|G| \ge m^2$.  
Then, $G$ fixes the point $\hat{Q}$. 
By Lemma \ref{fixed}, $G$ is a cyclic group. 
Therefore, $G$ is a cyclic group of order $m^2$. 
However, the cyclic group of order $m^2$ has a unique subgroup of order $m$. 
This is a contradiction.   
We have $m=1$. 
\end{proof}

For immersed curves, we have the following. 

\begin{lemma} \label{unramified2}
Assume that $\hat{C}_0=\hat{C}$, that is, the morphism $r: \hat{C} \rightarrow \mathbb{P}^2$ is unramified. 
Let $\hat{Q} \in \hat{C}$ and $Q=r(\hat{Q})$. 
If $P_1$ and $P_2$ are distinct Galois points, and $\hat{Q} \in \hat{C}$ is a common ramification point for $\hat{\pi}_{P_1}$ and for $\hat{\pi}_{P_2}$, then $P_1 \in \mathbb{P}^2 \setminus C$ and $Q=P_2$, or $P_2 \in \mathbb{P}^2 \setminus C$ and $Q=P_1$. 
\end{lemma}

\begin{proof} 
Assume that $Q \ne P_1, P_2$. 
Since $\hat{C}_0=\hat{C}$ and $\hat{Q} \in \hat{C}$ is a ramification point for $\hat{\pi}_{P_1}$ and $\hat{\pi}_{P_2}$, by Lemma \ref{index}(2), points $P_1$ and $P_2$ are contained in the tangent line $T_{\hat{Q}}C \subset \mathbb{P}^2$ at $\hat{Q}$. 
Therefore, $Q \in \overline{P_1P_2}$. 
However, by Lemmas \ref{unramified1} and \ref{index}(2), $Q \not\in \overline{P_1P_2}$. 
This is a contradiction. 
We have that $Q=P_1$ or $P_2$. 

Assume that $Q=P_2$. 
Since $\hat{Q}$ is a ramification point for $\hat{\pi}_{P_1}$, by Lemma \ref{index}(2), $P_1 \in T_{Q}C=T_{P_2}C$. 
According to \cite[Lemma 2.5]{fukasawa2} and Lemma \ref{index}, the point $P_1$ is not inner Galois.  
\end{proof}

We prove main theorems. 

\begin{proof}[Proof of Theorem \ref{main1}] 
Assume by contradiction that there exist three Galois points. 
Let $\hat{Q} \in \hat{C}$ with $Q=r(\hat{Q})$. 
We show that there exists a line passing through $Q$ with a linear polynomial $h$ defining it such that ${\rm ord}_{\hat{Q}}r^*h=1$. 
We can assume that $Q$ is not Galois. 
If there exist Galois points $P_1$ and $P_2$ such that $\overline{P_1P_2} \ni Q$, then we have the claim, by Lemmas \ref{unramified1} and \ref{index}(2). 
Therefore, we can assume that lines $\overline{PQ}$ are different for each Galois points $P$. 
Let $P_1, P_2$ and $P_3$ be Galois points, and let $h_1, h_2$ and $h_3$ be defining polynomials of $\overline{P_1Q}, \overline{P_2Q}$ and $\overline{P_3Q}$ respectively. 
Since the linear system associated with $r: \hat{C} \rightarrow \mathbb{P}^2$ is of dimension three, the values ${\rm ord}_{\hat{Q}}r^*h \ge 1$ for all lines $h=0$ passing through $Q$ have two possibilities (\cite[p.218]{hkt}, \cite[p.3]{stohr-voloch}).
Therefore, ${\rm ord}_{\hat{Q}}r^*h_i={\rm ord}_{\hat{Q}}r^*h_j$ for some $i \ne j$. 
By Lemma \ref{unramified1}, we have ${\rm ord}_{\hat{Q}}r^*h_i={\rm ord}_{\hat{Q}}r^*h_j=1$. 
\end{proof}

\begin{proof}[Proof of Theorem \ref{main2}] 
Assertion (d) is nothing but Proposition \ref{automorphism bound}. 

Assume that $\delta(C) \ge 3$. 
By Theorem \ref{main1}, $r: \hat{C} \rightarrow \mathbb{P}^2$ is unramified. 
By Lemma \ref{unramified2}, if $\hat{Q}$ is a ramification point for $\hat{\pi}_P$ from an inner Galois point $P$, then $\hat{Q}$ is not a ramification point for any other inner Galois point. By Fact \ref{Galois covering}(3) and the assumption that $d-1$ is not divisible by two, the ramification index $e_{\hat{Q}} \ge 3$, for each Galois point $P$ and a ramification point $\hat{Q} \in \hat{C}$ for $\hat{\pi}_P$.  
Let $m(P):=\min \{e_{\hat{Q}} \ | \ \hat{Q} \in \hat{C}, e_{\hat{Q}} \ge 3\}$ for each Galois point $P$, and let $m:=\min \{m(P) \ | \ P \mbox{ is Galois}\}$.   
Then, $m$ divides $d-1$. 
By the Riemann--Hurwitz formula, we have 
$$ 2g-2+2(d-1)=\sum_{\hat{Q} \in \hat{C}}(e_{\hat{Q}}-1). $$
Then, 
\begin{eqnarray*}
\sum_{\hat{Q} \in \hat{C}, \ e_{\hat{Q}} \ge 3}(e_{\hat{Q}}-2)&=&2g-2+2(d-1)+(-1)\times \#\{\hat{Q} \in \hat{C} \ | \ e_{\hat{Q}} \ge 3\} \\
&\ge& 2g-2+2(d-1)-\frac{1}{m-1}(2g-2+2(d-1)) \\
&=& \frac{m-2}{m-1}(2g-2+2(d-1)). 
\end{eqnarray*} 
Using Lemma \ref{index}, for each Galois point, we need at least 
$$ \frac{m-2}{m-1}(2g-2+2(d-1))$$
flexes. 
By Fact \ref{flexes}, we have the inequality
$$ \delta(C) \frac{m-2}{m-1}(2g-2+2(d-1)) \le 3(2g-2)+3d, $$
and hence, 
$$ \delta(C) \le \frac{m-1}{m-2} \times \frac{3(2g-2+d)}{2g-2+2(d-1)} <\frac{3(m-1)}{m-2}. $$
If $m=3$, then $\delta(C) <6$.
If $m \ge 5$, then $\delta(C) <4$. 
We consider the case $m=3$. 
Then, 
$$ \frac{m-1}{m-2} \times \frac{3(2g-2+d)}{2g-2+2(d-1)} < 4 \ \ (\mbox{resp. } <5) $$
if and only if 
$$ g < \frac{d-2}{2} \ \ (\mbox{resp. } g<2d-4). $$

For outer Galois points, we have 
$$ \delta'(C) \le \frac{m-1}{m-2} \times \frac{3(2g-2+d)}{2g-2+2d} <\frac{3(m-1)}{m-2}. $$
Therefore, we have $\delta'(C) \le 3$ (resp. $\delta'(C) \le 5$) if $m \ge 5$ (resp. if $m=3$).
Under the assumption that $m=3$, 
$$ \frac{m-1}{m-2} \times \frac{3(2g-2+d)}{2g-2+2d} < 4 \ \ (\mbox{resp. } <5)$$
if and only if 
$$ g < \frac{d+2}{2}  \ \ (\mbox{resp. } g<2d+1). $$
We complete the proof. 
\end{proof}


\begin{thebibliography}{100} 
\bibitem{duyaguit-miura} C. Duyaguit and K. Miura, On the number of Galois points for plane curves of prime degree, Nihonkai Math. J. {\bf 14} (2003), 55--59. 
\bibitem{fukasawa1} S. Fukasawa, Classification of plane curves with infinitely many Galois points, J. Math. Soc. Japan {\bf 63} (2011), 195--209. 
\bibitem{fukasawa2} S. Fukasawa, An upper bound for the number of Galois points for a plane curve, in ``Topics in Finite Fields,'' Contemp. Math. {\bf 632}, Amer. Math. Soc., 2015, pp. 111--119.  
\bibitem{hkt} J. W. P. Hirschfeld, G. Korchm\'{a}ros and F. Torres, Algebraic Curves over a Finite Field, Princeton Univ. Press, Princeton (2008).  
\bibitem{miura} K. Miura, Galois points on singular plane quartic curves, J. Algebra {\bf 287} (2005), 283--293. 
\bibitem{miura-yoshihara} K. Miura and H. Yoshihara, Field theory for function fields of plane quartic curves, J. Algebra {\bf 226} (2000), 283--294. 
\bibitem{stichtenoth} H. Stichtenoth, Algebraic Function Fields and Codes, Universitext, Springer-Verlag, Berlin (1993).
\bibitem{stohr-voloch} K.-O. St\"{o}hr and J. F. Voloch, Weierstrass points and curves over finite fields, Proc. London Math. Soc. (3) {\bf 52} (1986), 1--19. 
\bibitem{yoshihara1} H. Yoshihara, Function field theory of plane curves by dual curves, J. Algebra {\bf 239} (2001), 340--355. 
\bibitem{yoshihara2} H. Yoshihara, Galois points for plane rational curves, Far east J. Math. {\bf 25} (2007), 273--284; Errata, ibid. {\bf 29} (2008), 209--212. 
\bibitem{open} H. Yoshihara and S. Fukasawa, List of problems, available at: \\ http://hyoshihara.web.fc2.com/openquestion.html
\end{thebibliography}
\end{document}